\documentclass{amsart}
\usepackage{amsfonts,amsmath}

\usepackage[psamsfonts]{amssymb}

\theoremstyle{plain}
\newtheorem{theorem}{Theorem}[section]
\newtheorem{lemma}[theorem]{Lemma}
\newtheorem{proposition}[theorem]{Proposition}

\theoremstyle{definition}

\theoremstyle{remark}

\begin{document}

\title[Generating function for $\Delta$--Meixner--Sobolev orthogonal polynomials]{A
generating function for  non--standard orthogonal polynomials involving differences: the Meixner
case}

\author[J. J. Moreno-Balc\'{a}zar]{Juan J. Moreno-Balc\'{a}zar}
\address{Juan J. Moreno-Balc\'{a}zar. Departamento de Estad\'{\i}stica y Matem\'{a}tica Aplicada\\
Instituto Carlos I de F\'{\i}sica Te\'{o}rica y Computacional\\
Universidad de Almer\'{\i}a. Almer\'{\i}a. Spain.} \email{balcazar@ual.es}
\author[T. E. P\'{e}rez]{Teresa E. P\'{e}rez}
\address{Teresa E. P\'{e}rez. Departamento de Matem\'{a}tica Aplicada\\
Instituto Carlos I de F\'{\i}sica Te\'{o}rica y Computacional\\
Universidad de Granada. Granada. Spain.} \email{tperez@ugr.es}
\author[M. A. Pi\~{n}ar]{Miguel A. Pi\~{n}ar}
\address{Miguel A. Pi\~{n}ar. Departamento de Matem\'{a}tica Aplicada\\
Instituto Carlos I de F\'{\i}sica Te\'{o}rica y Computacional\\
Universidad de Granada. Granada. Spain.}
\email{mpinar@ugr.es}

\thanks{Partially supported by Ministerio de Educaci\'{o}n y Ciencia
        (MEC) of Spain and by the European Regional Development Fund
        (ERDF) through the grant MTM 2005--08648--C02, and Junta de
        Andaluc\'{\i}a, G. I. FQM0229 and Excellence Projects FQM481 and P06--FQM--01735.}

\begin{abstract}
 In this paper we deal with a family of non--standard
 polynomials orthogonal with respect to an inner product involving
 differences. This type of inner product is the so--called
 $\Delta$--Sobolev inner product. Concretely, we consider the case
 in which both measures appearing in the inner product
 correspond to the Pascal distribution (the orthogonal polynomials associated to
 this distribution are known as Meixner polynomials). The aim of this work is to obtain a generating
 function for the $\Delta$--Meixner--Sobolev orthogonal
 polynomials and, by using a limit process, recover a generating
 function for Laguerre--Sobolev orthogonal polynomials.
\end{abstract}

\subjclass[2000]{Primary 33C47, Secondary 42C05}

\keywords{Meixner polynomials, non--standard orthogonality,
generating function}

\maketitle

\section{Introduction}

Classical Meixner  polynomials $m_n(x;\beta,c)$, or Meixner
polynomials of the first kind according to the terminology used in
\cite{Ch}, are those polynomials associated to the Pascal
distribution
$$\mu(k)=\frac{c^{k} (\beta)_k}{k!}\, , \quad k=0,1,2\ldots, \quad  0<c<1, \quad \beta>0\, ,$$
that is, they are orthogonal with respect to the discrete standard
inner product
$$(f,g)=\int f\, g\,d\mu(k)=\sum_{k=0}^{+\infty} f(k)
g(k) \frac{c^{k} (\beta)_k}{k!}\, ,
$$
and they satisfy a second
order difference equation (see, for instance, \cite{nik-uva}). Some
authors in the literature use the distribution $(1-c)^{\beta}\mu(k)$
to have a probability measure. In this paper, following the
classical texts \cite{Ch} and \cite{nik-uva} or, more recently
\cite{ism}, we will consider  the discrete measure $\mu(k)$.

In \cite{AGM2}, a generalization of the above inner product is
introduced. The authors consider the $\Delta$--Sobolev inner product
\begin{equation} \label{sip}
(f,g)_S = \sum_{k=0}^{+\infty} f(k) g(k) \frac{c^{k} (\beta)_k}{k!}
+ \lambda \sum_{k=0}^{+\infty} \Delta f(k) \Delta g(k) \frac{c^{k}
(\beta)_k}{k!}
\end{equation}
with $\beta > 0, 0<c<1, \,  \lambda > 0\, ,$ and where $\Delta$ is
the usual forward difference operator defined by $\Delta
f(k)=f(k+1)-f(k)\, .$

As we can observe, (\ref{sip}) is a non--standard inner product ,
that is, $(xf,g)_S\ne (f,xg)_S\, .$ Thus, the corresponding sequence
of orthogonal polynomials does not satisfy a three--term recurrence
relation, and in general, the nice algebraic and differential
properties of standard orthogonal polynomials do not hold any more.

 We denote by  $\{S_n\}$ the
sequence of polynomials orthogonal with respect to \eqref{sip},
normalized by the condition that $S_n(x)$ and the Meixner polynomial
$m_n(x;\beta,c)$ have the same leading coefficient
($n=0,1,2,\ldots$). The polynomials $S_n(x)$ are the so--called
$\Delta$--Meixner--Sobolev orthogonal polynomials.  As we have
already mentioned, the polynomials $S_n(x)$ were introduced in
\cite{AGM2} where several algebraic and difference relations between
the families of polynomials $S_n(x)$ and $m_n(x;\beta,c)$ were
established. Asymptotic results for $S_n(x)$ when $n \to +\infty$,
have been obtained in \cite{AGMB1}.

The main goal of this paper is to obtain a generating function for
the polynomials $S_n(x)\, .$ Furthermore, we will be able to recover
the results obtained in \cite{mp} for Laguerre--Sobolev orthogonal
polynomials, that is, using a limit process we obtain the generating
function for Laguerre--Sobolev orthogonal polynomials from the
generating function for $\Delta$--Meixner--Sobolev orthogonal
polynomials. Thus,  we are in some sense working in one of the
direction pointed out in the recent survey about Sobolev orthogonal
polynomials on unbounded supports \cite{mmb} (second item of Section
4).

The structure of the paper is as follows: in Section 2 we state some
well--known results on classical Meixner polynomials which will be
used along the paper. Section 3 gives the basic relations on
$\Delta$--Meixner--Sobolev polynomials. In particular, it is shown
that a generating function for the $\Delta$--Meixner--Sobolev
polynomials can be reduced to a generating function involving  the
classical Meixner polynomials (Proposition \ref{serie}). In Section
4 a generating function for $\Delta$--Meixner--Sobolev polynomials
is derived. The main results are stated in Theorem \ref{beta1} and
\ref{mth}. Finally, in Section 5 we recover the generating function
for Laguerre--Sobolev orthogonal polynomials obtained in \cite{mp}.

\section{Classical Meixner Polynomials}

Let $\beta$, and $c$ be real numbers such that $c \neq 0, 1$, and
$\beta \neq 0, -1, -2, \ldots$ It is well known that classical
Meixner polynomials $m_n(x;\beta,c)$ can be defined by their
explicit representation in terms of the hypergeometric function
${}_2F_1$ (see, for instance, \cite[p. 175--177]{Ch} where a
different normalization is used),
\begin{equation} \label{exprep}
m_n(x;\beta,c) = \frac{(\beta)_n}{n!} {}_2F_1(-n,-x;\beta; 1 -
c^{-1})= \frac{(\beta)_n}{n!}\, \sum_{k=0}^{n} \binom{n}{k}
\frac{(-x)_k}{(\beta)_k} \left( \frac{1}{c}-1\right)^k,
\end{equation}
where $(a)_n$ denotes the usual Pochhammer symbol,
$$\quad (a)_0=1, \quad (a)_n = a (a+1) \cdots (a+n-1), \quad n\ge 1.$$
Observe that (\ref{exprep}) provides $m_n(x;\beta,c)$ as a
polynomial of exact degree $n$ with leading coefficient
\begin{equation}\label{lc}\frac{1}{n!}\, \left(1 -
\frac{1}{c}\right)^n.\end{equation}
If $\beta > 0$ and $0 < c < 1$, classical Meixner polynomials are
orthogonal with respect to the inner product,
\begin{equation} \label{mip}
(f,g) = \sum_{k=0}^{+\infty} f(k) g(k) \frac{c^{k}
(\beta)_k}{k!},
\end{equation}
and then,
\begin{equation} \label{norm}
\sum_{k=0}^{+\infty}  \left(m_n(k;\beta,c)\right)^2 \frac{c^{k}
(\beta)_k}{k!} = \frac{(\beta)_n}{n!\, c^{n}\,  (1-c)^{\beta}},
\quad n=0, 1, 2, \ldots
\end{equation}

\bigskip

Simplifying expression (\ref{exprep}), we get
\begin{equation} \label{mex-gen}
m_n(x;\beta,c) = \sum_{k=0}^{n}\frac{(\beta+k)_{n-k}}{k! \,
(n-k)!}\, (-x)_k \left(\frac{1}{c}-1\right)^k.\end{equation}

Observe that, for every value of the parameter $\beta$, expression
(\ref{mex-gen}) defines a polynomial of exact degree $n$, and
leading coefficient (\ref{lc}). In this way,  we can define Meixner
polynomials for all $\beta\in\mathbb{R}$.

Very simple  manipulations of the explicit representation
(\ref{mex-gen}) show that the main algebraic properties of the
classical Meixner polynomials still hold for the general case $\beta
\in \mathbb{R},$ and $c\in \mathbb{R}\setminus\{0,1\}$, although the
orthogonality given in (\ref{mip}) holds only for $\beta>0$ and
$0<c<1\,  .$ In particular, for $n\ge 1$,
 Meixner polynomials
satisfy a three--term recurrence relation
\begin{eqnarray*}
c\,(n+1)\, m_{n+1}(x;\beta,c) &=& \left[x\,(c-1) + \beta\, c
 + n\,(c+1)\right]\, m_n(x;\beta,c) -\\
&~ & -(n + \beta - 1 )\, m_{n-1}(x;\beta,c),
\end{eqnarray*}
with the initial conditions $m_{-1}(x;\beta,c)=0$, and
$m_0(x;\beta,c)=1$.

Moreover, the following relations are satisfied:
\begin{eqnarray}
&~& m_n(x;\beta,c) - m_{n-1}(x;\beta,c) = m_{n}(x;\beta-1,c), \label{der2}\\
&~& \Delta \left[m_n(x;\beta,c) - m_{n-1}(x;\beta,c) \right] =
\frac{c-1}{c} m_{n-1}(x;\beta,c). \label{der1}
\end{eqnarray}

\bigskip

The generating function for classical Meixner polynomials plays an
important role in this work. This generating function can be found,
for instance, in \cite[p. 176]{Ch} or \cite[p. 175]{ism}. Here, we
give an elementary proof of this result for general values of the
parameter $\beta$.

\begin{lemma} \label{leglag}
For $|\omega| < c < 1$ and $\beta \in \mathbb{R}\, ,$  we have
\begin{equation} \label{genlag}
\sum_{n=0}^{+\infty}\, m_n(x;\beta,c)\, \omega^n =
\left(1-\frac{\omega}{c}\right)^{x}\,
\left(1-\omega\right)^{-x-\beta}
\end{equation}
\end{lemma}

\begin{proof} From (\ref{mex-gen}), we get
\begin{align*}
\sum_{n=0}^{+\infty}\, m_n(x;\beta,c)\,\omega^n &= \sum_{k=0}^{+\infty}
\frac{(-x)_k}{k!} \left(\left(\frac{1}{c}-1\right) \omega
\right)^k\,\sum_{n=k}^{+\infty}
\frac{(\beta+k)_{n-k}}{(n-k)!} \, \omega^{n-k}\\
&=\sum_{k=0}^{+\infty} \frac{(-x)_k}{k!}
\left(\left(\frac{1}{c}-1\right) \omega
\right)^k\,\sum_{j=0}^{+\infty} \frac{(\beta+k)_j}{j!}\, \omega^{j}.
\end{align*}
Finally, using the well--known formula
\begin{equation} \label{f-han}
\sum_{j=0}^{+\infty}
\frac{(\alpha)_j}{j!}\,\omega^{j}=(1-\omega)^{-\alpha}, \quad
|\omega|<1,
\end{equation}
we obtain
\begin{align*}
\sum_{n=0}^{+\infty} m_n(x;\beta,c)\, \omega^n &= \sum_{k=0}^{+\infty} \frac{(-x)_k}{k!}
\left(\left(\frac{1}{c}-1\right) \omega \right)^k\, (1-\omega)^{-\beta-k} \\
&=(1-\omega)^{-\beta}\sum_{k=0}^{+\infty} \frac{(-x)_k}{k!}
\left(\frac{\left(\frac{1}{c}-1\right)\omega}{1-\omega}  \right)^k\\
&=(1-\omega)^{-\beta}\left(1-\frac{\left(\frac{1}{c}-1\right)\omega}{1-\omega}\right)^x=
\left(1-\frac{\omega}{c}\right)^{x}
\left(1-\omega\right)^{-x-\beta},
\end{align*}
for $|\omega|<c<1\, .$
\end{proof}

We  want to remark that, in this paper, we will use the previous
Lemma for $\beta>-1.$

\section{$\Delta$--Meixner--Sobolev orthogonal polynomials}

Let $\{S_n\}$ denote the sequence of polynomials orthogonal with
respect to the $\Delta$--Sobolev inner product
\begin{equation} \label{inpr}
(f,g)_S = \sum_{k=0}^{+\infty} f(k) g(k) \frac{c^{k} (\beta)_k}{k!}
+ \lambda \sum_{k=0}^{+\infty} \Delta f(k) \Delta g(k) \frac{c^{k}
(\beta)_k}{k!},\,
\end{equation}
with $\beta > 0, 0<c<1, \lambda > 0$. The polynomials $\{S_n\}$
are the so--called $\Delta$--Meixner--Sobolev orthogonal
polynomials, and they are normalized by the condition that the
leading coefficient of $S_n(x)$ equals the leading coefficient of
$m_n(x;\beta,c)$, $n\ge0$. Observe that $S_0(x) = m_0(x;\beta,c)$,
and $S_1(x) = m_1(x;\beta,c)$.

\bigskip

  The following result is obtained in \cite{AGM2}.

\begin{lemma}
There exist positive constants $a_n$ depending on $\beta, c$ and $\lambda$,
such that
\begin{equation} \label{twotwo}
m_n(x;\beta,c) - m_{n-1}(x;\beta,c) =
S_n(x) - a_{n-1} S_{n-1}(x), \quad n \ge 1.
\end{equation}
\end{lemma}

\begin{proof} Put
$$
m_n(x;\beta,c) - m_{n-1}(x;\beta,c) = m_n(x;\beta-1,c) =
S_n(x) + \sum_{i=0}^{n-1} c_i^{(n)} S_i(x).
$$
Then
$$
 c_i^{(n)} (S_i,S_i)_S = (m_n - m_{n-1}, S_i)_S.
$$
Applying \eqref{der2}, \eqref{der1}, and \eqref{inpr} to the
right--hand side, we obtain
$$
c_i^{(n)} = 0, \quad 0 \le i \le n-2,
$$
and
\begin{equation*}\begin{split}
 c_{n-1}^{(n)} (S_{n-1},S_{n-1})_S &=
 - \sum_{k=0}^{+\infty} m_{n-1}(k;\beta,c) S_{n-1}(k)\frac{(\beta)_k \, c^{k}}{k!}\\
 &= - \sum_{k=0}^{+\infty} \left( m_{n-1}(k;\beta,c) \right)^2 \frac{(\beta)_k \, c^{k}}{k!}.
\end{split}\end{equation*}
\end{proof}

\bigskip

The following recurrence relation for the coefficients $\{a_n\}$ in
\eqref{twotwo} is also obtained in \cite{AGM2}. Here, we write this
recurrence relation in an analogous form useful for our purposes.

\begin{lemma} \label{relrec}
The sequence $\{a_n\}_{n}$ in \eqref{twotwo} satisfies
\begin{equation*} \label{fc1}
a_{n} = \frac{n+\beta-1}{n+\beta-1 +
\left(1+\lambda\left(1-\frac{1}{c} \right)^2  \right) c\, n - c\, n \,
a_{n-1}}, \quad n \ge 1,
\end{equation*}
with
$$
a_0 = 1.
$$
\end{lemma}
\begin{proof} Write
\begin{eqnarray*}
R_0(x) &=& S_0(x), \\
R_n(x) &=& S_n(x) - a_{n-1}\, S_{n-1}(x), \quad n \ge 1,
\end{eqnarray*}
then for $n \ge 1$,
$$
(R_{n+1},R_n)_S + a_n (R_n,R_n)_S + a_n a_{n-1} (R_n,R_{n-1})_S = 0.
$$
After computing the $\Delta$--Sobolev inner products with
\eqref{norm}, \eqref{der1}, \eqref{inpr}, and \eqref{twotwo}, we
obtain \eqref{fc1}, for $n \ge 1$.

Finally, since $S_0(x) =m_{0}(x;\beta,c)$, and $S_1(x) = m_{1}(x;\beta,c)$, relation
\eqref{twotwo} implies $a_0 = 1$.
\end{proof}

\bigskip

In order to simplify the notations, from now on, we will denote by
\begin{equation*} \label{eta} \eta: = 1+\lambda\left(1-\frac{1}{c}
\right)^2 > 1.\end{equation*} Then, relation \eqref{fc1} reads
\begin{equation} \label{fc}
a_{n} = \frac{n+\beta-1}{n+\beta-1 + \eta \, c \, n - c \, n \, a_{n-1}}, \quad
n \ge 1,
\end{equation}
with $ a_0 = 1.$

To derive a generating function for $\Delta$--Meixner--Sobolev
orthogonal polynomials, we need more information about the sequence
$\{a_n\}$. The asymptotic behavior of this sequence was established
in \cite[Prop. 5]{AGMB1}. Again, we introduce this result in an
adequate form useful for our objectives, and  we also  give an
alternative and elemental proof.

\begin{lemma} \label{conv}
The sequence $\{a_n\}$ is convergent, and
\begin{equation*} \label{convan}
{a} = \lim_{n \to \infty} a_n = \frac{1+ \eta c-\sqrt{(1 + \eta c)^2 - 4c}}{2c},
\end{equation*}
is the smallest root of the equation
\begin{equation*} \label{eqconvan}
c \, z^2 - \left(1 +  \eta c \, \right)  z + 1 = 0.
\end{equation*}
\end{lemma}

\begin{proof}
First, we observe that a simple induction argument applied on
Lemma~\ref{relrec} gives $0 < a_n \le 1$, for all $n \ge 0$.

Suppose that $a = \displaystyle{\lim_{n \to +\infty} a_n}$ exists,
then \eqref{fc} implies
\begin{equation}\label{fca}
a = \frac{1}{1 + \eta c  - c  a} ,
\end{equation}
that is, $a$ is a solution of the equation
$$
c z^2 - \left(1 + \eta c\right)  z + 1 = 0.
$$
Since $a_n \le 1$ for all $n \ge 0$, we have ${a} \le 1$. Hence
$$
{a} = \frac{1+ \eta c-\sqrt{(1 + \eta c)^2 - 4c}}{2c} < 1.
$$
Now, we prove that $\{a_n\}$ is indeed convergent to $a$. With
\eqref{fc} and \eqref{fca}, we have
$$
\frac{1}{a_n} - \frac{1}{a} = \eta \, c \left(\frac{n}{n+\beta -1}-
1\right) - c  \left(\frac{n}{n+\beta -1} a_{n-1} -a \right).
$$
Then, using $0 < a_{n-1} \le 1$, and $0< a \le 1$, we get
\begin{eqnarray*}
\lefteqn{| a_n - a | = | a_n |\,| a |\, |\frac{1}{a_n} - \frac{1}{a} | <} \\
&<& \eta \, c
\left|\frac{n}{n+\beta -1}- 1\right| + c \, \left|\frac{n}{n+\beta -1}\right| \, |a_{n-1}
-a| + a\,
\left|\frac{\beta-1}{n+\beta -1}\right|.
\end{eqnarray*}
Hence
$$
\limsup | a_n - a | \le c \limsup | a_{n-1} - a |.
$$
Since $c < 1$, the lemma follows.
\end{proof}

From the sequence $\{a_n\}$ we construct a sequence $\{q_n(\eta)\}$
of polynomials in $\eta\, .$

\begin{lemma} \label{qus}
Define the sequence $\{q_n(\eta)\}$ by
$$
q_0(\eta) = 1,\quad q_{n+1}(\eta) = \frac{q_n(\eta)}{a_{n}},\quad n
\ge 0.
$$
Then $q_n(\eta)$, for $n\ge 1$, is a polynomial in $\eta$  (and
therefore in $\lambda$) such that $\deg q_n = n-1$, satisfying the
three--term recurrence relation
\begin{equation} \label{3trr}
(n+\beta -1)q_{n+1}(\eta) = \left(n + \beta - 1 + \eta \, c \, n\right)
q_n(\eta) - c \, n \, q_{n-1}(\eta), \quad n \ge 1,
\end{equation}
with initial conditions $q_0(\eta) = q_1(\eta) = 1$.
\end{lemma}

\begin{proof}
The recurrence relation \eqref{3trr} is just relation \eqref{fc}
rewritten in terms of $q_n(\eta)$. Since $a_0 = 1$, then $q_1 =
1$, and thus \eqref{3trr} implies that, for $n\ge 1$, $q_n$ is a
polynomial in $\eta$ of degree $n-1$.
\end{proof}

Note that in the limit case $\lambda=0$, we have $S_n(x)=m_n(x;
\beta,c)$ for all $n=0,1,2, \ldots$. Therefore, $\eta=1$, $a_n=
q_n(\eta)=1$, for all $n=0,1,2,\ldots$, and $a=1.$

\bigskip

Next result shows that  the formal power series, (i.e, generating
function) for $\Delta$--Meixner--Sobolev orthogonal polynomials can
be reduced to  a formal power series involving Meixner polynomials.

\begin{proposition} \label{serie}
We have
\begin{equation} \label{eqser}
\sum_{n=0}^{+\infty} q_n(\eta) \, S_n(x) \, \omega^n = \frac{1}{1-\omega}
\sum_{n=0}^{+\infty} q_n(\eta) \, m_n(x;\beta-1,c) \, \omega^n.
\end{equation}
\end{proposition}

\begin{proof}
Equation \eqref{twotwo} gives
\begin{equation*} \label{eqtwoq}
q_n(\eta) \, m_n(x;\beta-1,c) = q_n(\eta) \, S_n(x) -
q_{n-1}(\eta) \, S_{n-1}(x),
\end{equation*}
and therefore
$$
q_n(\eta)\, S_n(x) = \sum_{k=0}^n q_k(\eta) \,  m_k(x;\beta-1,c).
$$
Thus, we have
\begin{align*}
&\sum_{n=0}^{+\infty} q_n(\eta) \, S_n(x) \,
\omega^n=\sum_{n=0}^{+\infty}\left[ \sum_{k=0}^{n}
q_k(\eta)\,m_k(x;\beta-1,c)\,  \right] \omega^{n} \\
&=\sum_{k=0}^{+\infty} q_k(\eta)\,m_k(x;\beta-1,c)
\sum_{n=k}^{+\infty}\omega^{n-k} =\sum_{n=0}^{+\infty} \omega^n
\sum_{k=0}^{+\infty} q_k(\eta)\,m_k(x;\beta-1,c)\,
\omega^k\\
&=\frac{1}{1-\omega}\,\sum_{k=0}^{+\infty}
q_k(\eta)\,m_k(x;\beta-1,c)\, \omega^k\, .
\end{align*}

\end{proof}

\section{Generating function for $\Delta$--Meixner--Sobolev polynomials}

In this section, we will obtain a generating function for
$\Delta$--Meixner--Sobolev orthogonal polynomials with $\beta >0$ by
means of Proposition \ref{serie}, where Meixner polynomials
$m_n(x;\beta-1,c)$ are considered. The general approach uses the
explicit expression for Meixner polynomials (\ref{exprep}), where
$\beta \in \mathbb{R}$, $\beta \neq 0, -1, -2,...$. Note that, in
the case $\beta=1$, Meixner polynomials $m_n(x;0,c)$ defined by
(\ref{mex-gen}) appear in Proposition \ref{serie}. Therefore, we
have  to distinguish $\beta = 1$ and $\beta \neq 1$. We begin with
the particular case $\beta=1$ due to their simplicity. From now on,
we will denote $$G_M(x,\omega,\lambda):=\sum_{n=0}^{+\infty}
q_n(\eta) \, S_n(x) \omega^n\, .$$

\subsection{Case $\beta=1$}

In this case the generating function  is stated in the following
theorem.

\begin{theorem} \label{beta1}
Let $\{S_n\}$ be the sequence of orthogonal polynomials associated
with the $\Delta$--Sobolev inner product \eqref{inpr}, with $\beta
= 1$, and normalized by the condition that the leading coefficient
of $S_n$ equals the leading coefficient of $m_n(x;1,c)$. Let
$\{q_n(\eta)\}$ be defined by the recurrence relation
\begin{equation} \label{tcheby}
q_{n+1}(\eta) = (1 + \eta \, c)q_n(\eta) - c \, q_{n-1}(\eta), \quad q_0(\eta) = q_1(\eta) = 1.
\end{equation}
Then, for $|\omega|< a \, c <1$,
\begin{equation} \label{fg-beta1}
 G_M(x,\omega,\lambda)=
\frac{1}{1- \omega}\left[ \gamma \left(1-\frac{\omega}{a
c}\right)^{x} \left(1-\frac{\omega}{a}\right)^{-x} + \delta(1-
\omega a)^{x} (1- \omega c a)^{-x}
 \right], \end{equation}
where
\begin{equation} \label{expbyg}
a = \frac{1+ \eta c-\sqrt{(1 + \eta c)^2 - 4c}}{2c}, \qquad \gamma = \frac{a-a^2c}{1-a^2c}, \qquad
    \delta = \frac{1-a}{1-a^2c}.
    \end{equation}
\end{theorem}
\begin{proof}
If $\beta = 1$ the second order difference equation \eqref{3trr}
is reduced to \eqref{tcheby} and, therefore we have
$$
q_n(\eta) = \frac{1}{1 - a^2 c} \left( (a- a^2 c)\frac{1}{a^n} + (1
- a)(a c)^{n}\right).
$$
Thus, the theorem follows from Proposition~\ref{serie} and
Lemma~\ref{leglag}.
\end{proof}

\noindent \textbf{Remark.} It is important to note that, in the
limit case $\lambda=0$, we recover the generating function for
classical Meixner polynomials (\ref{genlag}) from (\ref{fg-beta1}),
since in this sitation $q_n(\eta)=1$, for all $n=0,1,2,\ldots$,
$a=1$, $\gamma=1$, and $\delta=0\, .$

\subsection{Case $\beta \neq 1$}

Now, we suppose $\beta > 0$ and $\beta \neq 1$. We will deduce a
generating function for the polynomials $S_n(x)$ starting from
relation \eqref{eqser}. First, we need a generating function for the
polynomials $q_n(\eta)\, .$

\begin{lemma} \label{le.4.1}
Let $\beta > 0, \beta \neq 1$, and let $\{q_n(\eta)\}$ be the
sequence of polynomials defined by the recurrence relation
\eqref{3trr}. Put
\begin{equation} \label{eqF}
F(\omega) = \sum_{n=0}^{+\infty} q_n(\eta)\,(\beta-1)_n
\frac{\omega^n}{n!},
\end{equation}
with $|\omega|<a <1$. Then,
\begin{equation} \label{exprF}
F(\omega) = \left(1-\frac{\omega}{a}\right)^{-(\beta
-1)\gamma}(1-\omega c a)^{-(\beta -1)\delta},
\end{equation}
where $a$, $\gamma$ and $\delta$ are defined in (\ref{expbyg}).
\end{lemma}

\begin{proof}
Observe that the ratio test shows that the series in the
right--hand side of \eqref{eqF} is convergent if $|\omega|<a <1$.
To simplify, if we write
\begin{equation}
\label{def-h} h_n(\eta) = \frac{q_n(\eta) (\beta-1)_n}{n!}, \quad
n \geq 0,
\end{equation}
then
$$F(\omega) = \sum_{n=0}^{+\infty} h_n(\eta) \omega^n.
$$
From \eqref{3trr}, we obtain the recurrence relation for
$\{h_n(\eta)\}$ as follows
\begin{equation} \label{3trrhn}
(n+1) h_{n+1}(\eta)=\left[n(1+\eta c)+\beta-1\right] h_n(\eta) - c
(n+\beta-2)h_{n-1}(\eta), \quad n\ge 1,
\end{equation}
with $h_0(\eta) = 1, h_1(\eta) = \beta-1$.

Multiplying \eqref{3trrhn} times $\omega^n$, and summing over $n=
1,2, \ldots$,  we obtain
$$
F'(\omega)- h_1(\eta) = (1 + \eta \, c) \omega F'(\omega) + (\beta - 1)
(F(\omega)-h_0(\eta)) - c \omega^2 F'(\omega)- c (\beta-1) \omega
F(\omega),
$$
hence
$$
F'(\omega)\left[1-(1 + \eta c)\omega + c \omega^2\right] = (\beta - 1)
F(\omega)(1-c \omega),
$$
with $1 + \eta c = 1/a + ca$. Then, we get
$$
F'(\omega)\left(1-\frac{\omega}{a}\right)(1 - \omega c a) =
(\beta-1) F(\omega)(1-c \omega),
$$
and, therefore, we have
$$
\left\{ \begin{array}{l}
\displaystyle{\frac{F'(\omega)}{F(\omega)} = (\beta-1)\left(
\frac{\gamma/a}{1-\frac{\omega}{a}} + \frac{\delta
c a}{1 - \omega c a}\right),} \\
~\\
F(0) = h_0(\eta) = 1,
\end{array}
\right.$$  where $\gamma$ and $\delta$ are defined in (\ref{expbyg}).
Solving this initial value problem, we obtain (\ref{exprF}).
\end{proof}

\noindent \textbf{Remark.} Note that in the limit case $\lambda=0$,
we have $a=1$ and, therefore $\gamma=1$, and $\delta=0.$ Thus, we
deduce $F(\omega)=(1-\omega)^{-\beta+1}$.

\bigskip

Now, we have the necessary tools  to obtain a generating function
for $\Delta$--Meixner--Sobolev orthogonal polynomials with $\beta
\neq 1.$

\begin{theorem} \label{mth}
Let $\{S_n\}$ be the sequence of polynomials orthogonal with
respect to the $\Delta$--Sobolev inner product \eqref{inpr} with
$\beta \neq 1$, and normalized by the condition that the leading
coefficient of $S_n(x)$ equals the leading coefficient of
$m_n(x;\beta,c)$. Let $\{q_n(\eta)\}$ be defined by the recurrence
relation (\ref{3trr}). Then, for $|\omega|< a c <1$,
\begin{align} \label{fg-betan1}
 G_M(x,\omega,\lambda) &=
\frac{1}{1- \omega}  (1-c a \omega)^{-(\beta-1) \delta} \left(
1-\frac{\omega}{a} \right)^{-(\beta-1) \gamma}
\left(1-\frac{\omega}{a
c}\right)^{x} \left(1-\frac{\omega}{a}\right)^{-x} \nonumber  \\
 &~\times \, {}_2F_1\left(  -x, (\beta-1)\delta; \beta-1,
 \frac{\omega (c-1)(1-a^2c)}{(1-ca\omega)(ac-\omega)}\right),
\end{align}
where $a$, $\gamma$ and $\delta$ are defined in (\ref{expbyg}).
\end{theorem}
\begin{proof}
We start giving two expressions for $k$--th derivative of
$F(\omega)$ defined in (\ref{eqF}). First, taking into account
(\ref{def-h}), we have
\begin{equation} \label{derF-1}
F^{(k)}(\omega)=\sum_{n=k}^{+\infty} \frac{n!}{(n-k)!} \, h_n(\eta) \, \omega^{n-k}.\end{equation}
 On the other hand, from (\ref{exprF}), we get
\begin{align} \label{derF-2} F^{(k)}(\omega)&= \sum_{s=0}^k \binom{k}{s} \left[(1- c a \omega)^{-(\beta
-1)\delta} \right]^{(s)} \left[
\left(1-\frac{\omega}{a}\right)^{-(\beta -1)\gamma}
\right]^{(k-s)} \nonumber \\
&=(1- c a \omega)^{-(\beta
-1)\delta}\left(1-\frac{\omega}{a}\right)^{-(\beta
-1)\gamma} \nonumber \\
&~\times \sum_{s=0}^{k} (-1)^k k! \binom{-(\beta-1)\delta}{s}
\binom{-(\beta-1)\gamma}{k-s}\left( \frac{ca}{1-ca\omega}
\right)^s \left( \frac{1}{a-\omega} \right)^{k-s}.
\end{align}
Now, with (\ref{derF-1}) and the explicit representation of Meixner
polynomials (\ref{exprep}),
 we get
\begin{eqnarray*}
\lefteqn{\sum_{n=0}^{+\infty}q_n(\eta)\, m_n(x; \beta-1,c)\, \omega^n
=}\\
&=& \sum_{n=0}^{+\infty}q_n(\eta)\,\left[
\frac{(\beta-1)_n}{n!}\sum_{k=0}^n \binom{n}{k}\frac{(-x)_k}{(\beta-1)_k}
\left(\frac{1}{c}-1\right)^k\right] \, \omega^n  \\
&=&\sum_{k=0}^{+\infty}  \frac{(-x)_k}{k! \, (\beta-1)_k}\left(\frac{1}{c}-1\right)^k \, \omega^k
\sum_{n=k}^{+\infty} \frac{n!}{(n-k)!} \, h_n(\eta)\omega^{n-k}\\
&=& \sum_{k=0}^{+\infty} \frac{(-x)_k}{k! \, (\beta-1)_k}\left(\frac{1}{c}-1\right)^k \, \omega^k \,
F^{(k)}(\omega).
\end{eqnarray*}
Thus, from (\ref{derF-2}), we obtain
\begin{align*}
&\sum_{n=0}^{+\infty}\,q_n(\eta)\,m_n(x; \beta-1,c)\,\omega^n = (1- c a
\omega)^{-(\beta
-1)\delta} \left(1-\frac{\omega}{a}\right)^{-(\beta -1)\gamma}\\
&\times \sum_{k=0}^{+\infty} \left\{\frac{(-x)_k}{(\beta-1)_k}
\left(1-\frac{1}{c} \right)^k \omega^k \right.\\
&\qquad\times\left.\sum_{s=0}^k \binom{-(\beta-1)\delta}{s}
\binom{-(\beta-1)\gamma}{k-s}\left(
\frac{ca}{1-ca\omega} \right)^s \left( \frac{1}{a-\omega}
\right)^{k-s} \right\} \\
&=(1- c a \omega)^{-(\beta -1)\delta}
\left(1-\frac{\omega}{a}\right)^{-(\beta -1)\gamma}\\
&\times
\sum_{s=0}^{+\infty} \left\{ \binom{-(\beta-1)\delta}{s}\left(
1-\frac{1}{c} \right)^s \left( \frac{ca}{1-ca\omega} \right)^s
\omega^s \right. \\
&\qquad\times\left. \sum_{k=s}^{+\infty} \binom{-(\beta-1)\gamma}{k-s}
\frac{(-x)_k}{(\beta-1)_k} \left( 1-\frac{1}{c} \right)^{k-s}
\left( \frac{1}{a-\omega} \right)^{k-s} \omega^{k-s}\right\} \\
&=(1- c a \omega)^{-(\beta -1)\delta}
\left(1-\frac{\omega}{a}\right)^{-(\beta -1)\gamma}\\
&\times\sum_{s=0}^{+\infty} \left\{ \binom{-(\beta-1)\delta}{s}
\frac{(-x)_s}{(\beta-1)_s} \left( 1-\frac{1}{c} \right)^s \left(
\frac{ca}{1-ca\omega} \right)^s
\omega^s \right. \\
&\qquad\times\left. \sum_{m=0}^{+\infty} \binom{-(\beta-1)\gamma}{m}
\frac{(-x+s)_{m}}{(\beta-1+s)_m} \left( 1-\frac{1}{c} \right)^{m}
\left(\frac{1}{a-\omega} \right)^{m} \omega^{m}\right\},
\end{align*}
where in last equality we use $(-x)_{s+m}=(-x)_s (-x+s)_m$. If we
denote
\begin{equation}\label{omega1}
\omega_1:=\left( 1-\frac{1}{c}
\right)\frac{c a \omega}{1-c a \omega}=\frac{(c-1) a \omega}{1-c a
\omega},  \quad \omega_2:= \frac{\omega}{a-\omega}\left( 1-\frac{1}{c}
\right)=\frac{(c-1)\omega}{c(a-\omega)},
\end{equation}
the above expression yields
\begin{align*}
&\sum_{n=0}^{+\infty}q_n(\eta)\,m_n(x; \beta-1,c)\,\omega^n=(1- c a
\omega)^{-(\beta
-1)\delta} \left(1-\frac{\omega}{a}\right)^{-(\beta -1)\gamma}\\
&\times \sum_{s=0}^{+\infty} \frac{((\beta-1) \delta)_s
(-x)_s}{(\beta-1)_s}\, \frac{(-\omega_1)^s}{s!}
\sum_{m=0}^{+\infty} \frac{((\beta-1) \gamma)_m
(-x+s)_m}{(\beta-1+s)_m}\, \frac{(-\omega_2)^m}{m!}\\
&=(1- c a \omega)^{-(\beta
-1)\delta} \left(1-\frac{\omega}{a}\right)^{-(\beta -1)\gamma}\\
&\times \sum_{s=0}^{+\infty} \frac{((\beta-1) \delta)_s
(-x)_s}{(\beta-1)_s}\, \frac{(-\omega_1)^s}{s!}\,
{}_2F_1(s-x,(\beta-1)\gamma; s+\beta-1; -\omega_2) \\
&=(1- c a \omega)^{-(\beta -1)\delta}
\left(1-\frac{\omega}{a}\right)^{-(\beta -1)\gamma}
(1+\omega_2)^{-(\beta-1)\gamma}\\
&\times \sum_{s=0}^{+\infty} \frac{(-x)_s((\beta-1)
\delta)_s}{(\beta-1)_s}\, \frac{(-\omega_1)^s}{s!}\,
{}_2F_1\left(x+\beta-1, (\beta-1)\gamma ; s+\beta-1;
\frac{\omega_2}{\omega_2+1}\right),
\end{align*}
where in the last equality we have been able to apply the
Pfaff--Kummer transformation  (see, for instance, \cite[f.
(1.4.9)]{ism} or \cite[p. 425]{han})
$${}_2F_1(a,b; c;z)=(1-z)^{-b}\,{}_2F_1\left(c-a,b;
c;\frac{z}{z-1}\right)\, , \quad |z|<1\, ,$$ since $|\omega_2|<1$
for $|\omega|<ac\, .$

In order to simplify the above expression, we can observe that we
are in situation to apply formula (65.2.2) in \cite{han}, i.e.,
$$\sum_{k=0}^{+\infty} \frac{(a)_k (b)_k}{(c)_k}\frac{y^k}{k!}\,
_2F_1(c-a, c-b;c+k;z)=(1-z)^{a+b-c}\, _2F_1(a,b;c;z+y-zy),$$ since
$(\beta-1)(1-\delta)=(\beta-1)\gamma.$ Therefore, after some
simplifications, we get
\begin{align*}
&\sum_{n=0}^{+\infty}q_n(\eta)\, m_n(x; \beta-1,c)\, \omega^n=(1- c a
\omega)^{-(\beta -1)\delta}
\left(1-\frac{\omega}{a}\right)^{-(\beta -1)\gamma}(1+\omega_2)^x\\
&\times \, {}_2F_1\left( -x, (\beta-1)\delta; \beta-1;
\frac{\omega_2-\omega_1}{\omega_2+1} \right).
\end{align*}
Finally, using Proposition \ref{serie}, and the explicit expressions for $\omega_1$,
 and $\omega_2$ given in (\ref{omega1}), we obtain (\ref{fg-betan1}).
\end{proof}

\noindent \textbf{Remark.} In the limit case $\lambda=0$, we have
\begin{align*}
&\sum_{n=0}^{+\infty} m_n(x; \beta,c) \, \omega^n\\
&=\frac{1}{1-\omega} (1-\omega)^{-(\beta-1)} \left(
\frac{c-\omega}{c(1-\omega)}\right)^x \, _2F_1\left( -x,0;\beta-1;
\frac{-\omega(1+c^2)}{(1-c\omega)(c-\omega)} \right)\\
&= (1-\omega)^{-x-\beta} \left(1-\frac{\omega}{c} \right)^x,
\end{align*}
and we obtain again the generating function for Meixner polynomials.

\noindent \textbf{Remark.} Of course, the case $\beta = 1$ in
Theorem \ref{beta1} can be deduced from (\ref{fg-betan1}), since as
we can easily check
\begin{align*}
\lim_{\beta \to 1} {}_2F_1\left( -x, (\beta-1)\delta; \beta-1;
\frac{\omega_2-\omega_1}{\omega_2+1} \right) &= \gamma + \delta
\sum_{k=0}^{+\infty} \frac{(-x)_k}{k!} \left(
\frac{\omega_2-\omega_1}{1+\omega_2}\right)^k
\\
&=\gamma + \delta \left(1 -
\frac{\omega_2-\omega_1}{1+\omega_2}\right)^x
\end{align*}
and the result follows from the explicit expressions for $\omega_1$
and $\omega_2$.

\section{Generating function for Laguerre--Sobolev orthogonal
polynomials}

In this section, by using a limit process we will recover the
generating function for the Laguerre--Sobolev orthogonal polynomials
obtained in \cite{mp}. As it is well--known (see, for instance,
\cite[p. 177]{Ch}) there exists a limit relation between Meixner and
Laguerre orthogonal polynomials, namely
\begin{equation} \label{limit-ml}
\lim_{c\uparrow 1}c^n m_n^{(\alpha+1,c)}\left( \frac{x}{1-c}
\right)=L_n^{(\alpha)}(x)\, , \quad \alpha>-1\, ,
\end{equation}
where $L_n^{(\alpha)}(x)$ denotes the Laguerre polynomials with
leading coefficient $(-1)^n/n!$ orthogonal with respect to the inner
product $$(f,g)_L=\int_0^{+\infty} f(x)\, g(x)\,x^{\alpha}\,
e^{-x}dx\, .$$

In \cite[Prop. 4.4]{AGMB2} the authors give a formula which
extends the limit relation (\ref{limit-ml}) to the
$\Delta$--Sobolev case in the framework of $\Delta$--coherence. It
is important to note that
 in this paper Meixner polynomials are considered orthogonal
 with respect to the inner product $(1-c)^{\beta}(f,g)$ where
 $(f,g)$ is given in (\ref{mip}). Anyway, taking
 \begin{equation} \label{values}
 \beta=\alpha+1\quad \textrm{and} \quad
 \lambda=\frac{\tilde{\lambda}}{(1-c)^2}\, , \quad
 \tilde{\lambda}>0\, ,
 \end{equation}
 and using the same arguments as in \cite{AGMB2}, we can prove
 \begin{equation} \label{limit-msls}
\lim_{c\uparrow 1}c^n S_n\left( \frac{x}{1-c} \right)=S_n^{L}(x)\,
,
\end{equation}
where $\{S_n^{L} \}$ are the so--called Laguerre--Sobolev
polynomials with leading coefficient $(-1)^n/n!$ orthogonal with
respect to the inner product $$(f,g)_{\tilde{S}}=\int_0^{+\infty}
f(x)\, g(x)\,x^{\alpha}\, e^{-x}dx+\tilde{\lambda}\int_0^{+\infty}
f'(x)\, g'(x)\,x^{\alpha}\, e^{-x}dx \, .$$ Note that the values for
$\beta$ and $\lambda$ given in (\ref{values}) imply
$$\eta=1+\frac{\tilde{\lambda}}{c^2}\, .$$

In \cite{mp}, a generating function for  polynomials $S_n^{L}(x)$
was obtained. In fact, if we denote by $\{q_n^L(\tilde{\lambda})\}$
the sequence of polynomials defined by the recurrence relation
\begin{equation}
\label{rrq-l}
(n+\alpha)q_{n+1}^L(\tilde{\lambda})=\left[n(\tilde{\lambda}+2)+\alpha\right]
q_n^L(\tilde{\lambda})-n q_{n-1}^L(\tilde{\lambda})\, ,
\end{equation}
with $q_{0}^L(\tilde{\lambda})=q_{1}^L(\tilde{\lambda})=1\, ,$
 and
$$
G_L(x,\omega, \tilde{\lambda}):=\sum_{n=0}^{\infty}
q_{n}^L(\tilde{\lambda})S_n^L(x)\,\omega^n\, .
$$
Then, for $|\omega|<\tilde{a}<1\, ,$ we get (see Theorems 2.1 and
3.1 in \cite{mp})
\begin{itemize}
\item For $\alpha=0$
\begin{equation*} \label{fga0}
G_L(x,\omega,\tilde{\lambda})=\frac{1}{(1-\omega)(1+\tilde{a})}\left[
\exp \left(\frac{-x \omega \tilde{a}}{1-\omega \tilde{a}}
\right)+\tilde{a}\exp \left(\frac{-x
\omega/\tilde{a}}{1-\omega/\tilde{a}} \right) \right]
\end{equation*}

\item For $\alpha\neq 0$
\begin{align*} \label{fgan0}
G_L(x,\omega,\tilde{\lambda})&=\frac{1}{1-\omega}\, (1-\tilde{a}
\omega)^{\frac{-\alpha}{1+\tilde{a}}}\,
\left(1-\frac{\omega}{\tilde{a}} \right)^{\frac{\alpha\,
\tilde{a}}{1+\tilde{a}}}\,\exp\left(\frac{-x\omega/\tilde{a}}{1-\omega/\tilde{a}} \right)\\
& \times {}_1F_1\left( \frac{\alpha}{1+\tilde{a}}; \alpha; \frac{x
\omega (1-\tilde{a}^2)}{(\tilde{a}-\omega)(1-\omega \tilde{a})}
\right)\, ,
\end{align*}
\end{itemize}
where, in both cases,
$$\tilde{a}=\frac{\tilde{\lambda}+2-\sqrt{\tilde{\lambda}^2+4\tilde{\lambda}}}{2}\,
.$$

Using again the values for $\beta$ and $\tilde{\lambda}$ given in
(\ref{values}) and taking limits when $c\uparrow 1$ in (\ref{3trr})
we recover (\ref{rrq-l}) with the same initial conditions.
Therefore, we get  \begin{equation} \label{c1qs} \lim_{c\uparrow
1}q_n(\eta)= q_n^L(\tilde{\lambda})\, .\end{equation} Thus, using
(\ref{limit-msls}) and (\ref{c1qs}) we obtain \begin{align*}
\lim_{c\uparrow 1}G_M\left(\frac{x}{1-c}, c\, \omega,
\frac{\tilde{\lambda}}{(1-c)^2}\right)& = \lim_{c\uparrow
1}\sum_{n=0}^{+\infty}q_n\left(\eta\right)
S_n\left(\frac{x}{1-c}\right)\, c^n\,\omega^n \\
&=\sum_{n=0}^{+\infty}q_n^L(\tilde{\lambda})\, S_n^{L}(x)\,
\omega^n=G_L(x,\omega,\tilde{\lambda})\, , \end{align*} for
$|w|<\tilde{a}$ (note that $\lim_{c\uparrow 1}a=\tilde{a}$ with
$\beta$ and $\lambda$ given in (\ref{values})).

Therefore, we  claim that we have recovered the generating
functions for Laguerre--Sobolev orthogonal polynomials from the
generating functions for $\Delta$--Meixner--Sobolev orthogonal
polynomials.

\end{document}